\documentclass[reqno,11pt]{amsart}

\numberwithin{equation}{section}
\usepackage{amsmath}
\usepackage{esint} % permette di fare l'integrale tagliato
\usepackage{braket}
\usepackage{cancel}
\usepackage{amsthm}
\usepackage{epsfig}
\usepackage{psfrag}
\usepackage{graphicx}
\graphicspath{{figure/}}
\usepackage{graphpap,latexsym,epsf}
\usepackage{color}
\usepackage{amssymb,mathrsfs,enumerate}
\usepackage{mathtools}
\usepackage{subfigure}
\definecolor{citegreen}{rgb}{0,0.6,0}
\definecolor{refred}{rgb}{0.8,0,0}
\usepackage[colorlinks, citecolor=citegreen, linkcolor=refred]{hyperref}
\newcommand{\R}{\mathbb{R}}

\newcommand{\pa}{\partial}
\newcommand{\Om}{\Omega}
\newcommand{\ffi}{\varphi}

\newcommand{\ep}{\varepsilon}
\newcommand{\rmd}{{\rm d}}

\newcommand{\D}{\nabla}
\newcommand{\DD}{\nabla^2}
\newcommand{\De}{\Delta}

\newcommand{\na}{\nabla}
\newcommand{\nana}{\nabla^2}

\newcommand{\A}{L}
\newcommand{\B}{M}

\mathchardef\emptyset="001F

\definecolor{vgreen}{rgb}{0.1,0.5,0.2}
\definecolor{viola}{RGB}{85,26,139}
\newtheorem{theorem}{Theorem}[section]
\newtheorem{remark}{Remark}

\newtheorem{proposition}[theorem]{Proposition}

\newtheorem{lemma}[theorem]{Lemma}

\begin{document}

\title[Symmetry results for Serrin-type problems in doubly connected domains]
{Symmetry results for Serrin-type problems\\ in doubly connected domains}

\author[S.~Borghini]{Stefano Borghini}
\address{S.~Borghini, Universit\~a degli Studi di Milano--Bicocca, Via Roberto Cozzi 55, 20126 Milano, Italia}
\email{stefano.borghini@unimib.it}

\begin{abstract} 
%In this paper we study Serrin-type problems for the standard laplacian on ring-shaped domains.
%Our main tool will be a comparison technique introduced in~\cite{ABM}. In fact, one of the main purposes of the present paper will be that of exploring this strategy and its comparison with the well known moving plane method. More precisely, our aim is twofold: on the one hand, we will exploit the method directly in order to deal with a case that appears not to be achievable through the moving plane method, leading to a new rigidity result. On the other hand, we will also expand on the method, by means of a subtle application of the Pohozaev identity in combination with the isoperimetric inequality exploiting this to give an alternative, somehow more elementary, proof of a known result in the literature.
In this work, we employ the technique developed in~\cite{ABM} to prove rotational symmetry for a class of Serrin-type problems for the standard laplacian. We also discuss in some length how our strategy compares with the classical moving plane method.
\end{abstract}

\maketitle

\section{Introduction and statement of the main results}

%In this paper, we focus on the study of a classical overdermined problem for the standard Laplace operator.
In his pioneering paper~\cite{Serrin}, Serrin refined Alexandrov's method of moving planes~\cite{Alexandrov} and employed it to prove that a function defined in a bounded domain having constant nonzero laplacian and satisfying constant Dirichlet and Neumann boundary condition is necessarily rotationally symmetric. 
%More precisely, up to the addition of a constant, the solution has the form
%$$
%\Omega_0\,=\,B(0,R)\,,\qquad u_0\,=\,\frac{R^2-|x|^2}{2}\,,
%$$
%for some $R>0$. 
In fact, in~\cite{Serrin} it was shown that similar results hold more generally for a vast class of elliptic PDEs.
Serrin's result and the method of moving planes have given rise to a prolific field of research on overdetermined boundary value problems. Nowadays Serrin's technique has been employed to study analogous problems in a number of work, among which we mention~\cite{Barrios_Montoro_Sciunzi,
Brandolini_Nitsch_Salani_Trombetti,Gidas_Ni_Nirenberg,Ma_Liu}.
An alternative proof of Serrin's result was given by Weinberger~\cite{Weinberger} using a completely different method, essentially based on a comparison argument with the model solutions.
 Weinberger's method and in particular its use of a P-function has also experienced quite some success~\cite{Ciraolo_Vezzoni,Farina_Kawohl,
 Farina_Valdinoci,Garofalo_Lewis,Roncoroni}. For further insights on these methods and their applications, see also~\cite{Magnanini} and references therein.

The above mentioned papers all gave rise to interesting characterizations of solutions supported in a ball.
In this work, we are instead interested in characterizations of solutions supported in an annulus. Let us explain our setting in more details.
Let $E_i\Subset E_o\subset\R^2$ be two nonempty simply connected bounded domains, whose boundaries $\pa E_o=\Gamma_o$ and $\pa E_i=\Gamma_i$ are simple closed smooth curves, and let $\Omega=E_o\setminus \overline{E_i}$. By contruction, $\Omega$ is a bounded domain that is not simply connected, with boundary
$$
\pa\Omega\,=\,\Gamma_i\sqcup\Gamma_o\,.
$$
%In accordance with~\cite{ABM}, a domain $\Omega$ with this structure will be referred to as a {\em ring-shaped domain}.
%
%
%We are interested in studying pairs $(\Om,u)$, where $\Omega=E_o\setminus \overline{E_i}\subset\R^2$ is a ring-shaped domain and $u:\Omega\to\R$ is a solution to the following system
We are interested in studying pairs $(\Om,u)$, where $\Omega$ has the above structure and $u:\Omega\to\R$ is a solution to the following problem

\begin{equation}
\label{eq:problem}
\begin{dcases}
\Delta u=-2, & \mbox{in } \Om\,,\\
u=a\,,\ \ \frac{\pa u}{\pa\nu}=\alpha, & \mbox{on } \Gamma_i\,,
\\
u=b\,,\ \ \frac{\pa u}{\pa\nu}=\beta, & \mbox{on } \Gamma_o\,,
\\
\end{dcases}
\end{equation}
where $a,b,\alpha,\beta$ are constants and $\nu$ is the unit-normal pointing outside $\Omega$. 
This overdetermined boundary value problem has been well studied, especially in the case where $a>b$ and $\alpha\geq 0$. Under this assumption, the method of moving planes shows its strength and allows to prove some powerful rigidity results. A particularly nice one is the following, due to Reichel. 

\begin{theorem}[{\cite[Theorem~2]{Reichel2}}]
\label{thm:reichel}
Let $(\Omega,u)$ be a solution to problem~\eqref{eq:problem} such that $a>b$ and $b<u<a$ in $\Omega$. Then $\Omega$ is an annulus and $u$ is rotationally symmetric with $\pa u/\pa |x|<0$.
\end{theorem}

It is worth mentioning that Reichel's result is slightly stronger, as it also works for more general PDEs and it allows $E_i$ (and thus $\Gamma_i$) to be disconnected.
In~\cite{Sirakov}, Sirakov generalized previous
results~\cite{Aftalion_Busca,Alessandrini,
Willms_Gladwell_Siegel} and proved that Reichel's thesis remains true if one allows different values of $a$ and $\alpha$ on different connected components of $\Gamma_i$.
He was also able to replace the hypothesis $u<a$ with the weaker $\alpha\geq 0$.
We also remark that similar symmetry results have been proven to hold for more general families of elliptic operators. In~\cite{Sirakov} it is shown that the same result works for quasi-linear regular strongly elliptic operator. Similar results also hold for quasi-linear possibly degenerate elliptic operators~\cite{Alessandrini}, fully nonlinear operators~\cite{Reichel3} and for the fractional laplacian~\cite{Soave_Valdinoci}. For further symmetry results on this and related problems, see also~\cite{Barbu_Enache,Ciraolo_Vezzoni,Enciso_Peralta-Salas,Payne_Philippin,Reichel1}. 

On the other hand, it seems that less attention has been brought to the case $a<b$. One possible explaination is that the moving plane method works well when the model solution are monotonically decreasing, but seems to be harder to employ when the model solutions are monotonically increasing. 
In this work, our aim is to discuss this case and to prove the analogue of Theorem~\ref{thm:reichel}, namely:

\begin{theorem}
\label{thm:main}
Let $(\Omega,u)$ be a solution to problem~\eqref{eq:problem} such that $a<b$ and $a<u<b$ in $\Omega$. Then $\Omega$ is an annulus and $u$ is rotationally symmetric with $\pa u/\pa |x|>0$.
\end{theorem}

The strategy of the proof is based on the method developed in~\cite{ABM}. There, problem~\eqref{eq:problem} has been studied in the case $a=b=0$, using a new comparison strategy. %that somehow resembles the one of Weinberger. 
Of course, under that hypothesis, monotonicity is lost and there are some complications, especially concerning the behaviour of the solution near the set of the maxima of the function $u$. However, we will show that
the core of the method employed in~\cite{ABM} is perfectly suited to study the case discussed in this paper. In fact, we can adapt most of the argument in~\cite{ABM} to prove Theorem~\ref{thm:main}. 

Perhaps surprisingly, our method proves to be less effective in the more studied case $b<a$. We will try to explain along the work why this is the case. Nevertheless, we will be able to employ our method to prove a slightly weaker version of Theorem~\ref{thm:reichel} (namely, we will need to assume the additional condition $2a+\alpha^2\leq 2b+\beta^2$, see Theorem~\ref{thm:main2}). As we will see, the proof is also less elementary than the one for Theorem~\ref{thm:main}.
%, as it is based on a combination of the Pohozaev identity and the isoperimetric inequality. 
It does seem then that the moving plane method is still the one best suited to characterize rotationally symmetric solutions that are radially decreasing.
Nevertheless, our alternative proof may be of interest: it is novel and it expands on our method in ways that may find further applications in the future.

Let us conclude with some final comments on the above results and on future applications. First of all, one natural question would be whether similar characterizations can be achieved for model solutions that are not monotonic along the radial coordinate. In this case, the situation is less clear as there are negative results: in~\cite{Kamburov_Sciaraffia} it is proven that there exist solutions to~\eqref{eq:problem} with $\alpha=\beta$ that are not rotationally symmetric. In the case $a=b=0$, there are both positive (\cite[Theorem~B and Theorem~C]{ABM}) and negative results (\cite[Theorem~A]{ABM}).
Another natural question concerns whether our method can be applied to more general elliptic equations (in the same way as the Weinberger's comparison method~\cite{Weinberger} was then developed in~\cite{Farina_Kawohl,
 Farina_Valdinoci,Garofalo_Lewis}) or to other problems in Riemannian geometry. In fact, the most part of our arguments (with the notable exception of the Pohozaev identity) do not rely on the structure of the Euclidean space at all. As a couple of examples of further applications, this method has proven to be quite successful to characterize static spacetimes in General Relativity~\cite{Borghini,BCM,Borghini_Mazzieri_I,
 Borghini_Mazzieri_II} and somewhat similar techniques have been employed for other problems in General Relativity~\cite{Agostiniani_Mazzieri_Oronzio,
 Fogagnolo_Pinamonti} and for $p$-harmonic functions in manifolds with nonnegative Ricci curvature in~\cite{Agostiniani_Fogagnolo_Mazzieri,
 Fogagnolo_Mazzieri_Pinamonti}.
 
The paper is structured as follows. In Section~\ref{sec:setup} we show that for any acceptable choice of $a,b,\alpha,\beta$ there is a rotationally symmetric solution solving problem~\eqref{eq:problem}. This is necessary in order to start our method, as we need to select a model solution to compare with. Section~\ref{sec:exploit} is the heart of the paper: a crucial gradient estimate (Theorem~\ref{thm:gradient_estimate}) is introduced and exploited to obtain some important area bounds for the boundary components (Proposition~\ref{pro:area_bound_Gamma}). In Subsection~\ref{sub:proof_main} these two results are then used, in combination with a divergence theorem argument, to prove Theorem~\ref{thm:main}. Unfortunately, the method described in Subsection~\ref{sub:proof_main} does not work to prove Theorem~\ref{thm:reichel}. In Section~\ref{sec:poho} we then develop a new alternative argument, based on a combination of the Pohozaev identity and the isoperimetric inequality, leading to the proof of Theorem~\ref{thm:main2}, which is a weaker version of Theorem~\ref{thm:reichel}. Finally, in Section~\ref{sec:comments} we further comment on the complications that one has to deal with when studying the case $a>b$, trying to analyze where they come from and how one may try to overcome them.  
Ultimately, the purpose of this final section is that of understanding what are the limits of our method and how far they can be pushed.
%The purpose of this final section is that of showing what can go wrong with our method and how one may try to deal with complications.

\section{Setting up the comparison argument}
\label{sec:setup}

Our first aim is that of understanding whether there are some relations that are always in place between the constants $a,\alpha,b,\beta$. The main tool that will help us in this regard is the well known {\em Pohozaev identity}~\cite{Pohozaev}. In our framework, this formula has the following form
\begin{align*}
\int_{\Omega}4\,u\,d\mu\,&=\,\frac{1}{2}\int_{\pa\Omega}(4\,u+|\D u|^2)\Braket{X|\nu}\,d\sigma
\\
&=\,\frac{1}{2}(4\,b+\beta^2)\int_{\Gamma_o}\Braket{X|\nu}\,d\sigma\,+\,\frac{1}{2}(4\,a+\alpha^2)\int_{\Gamma_i}\Braket{X|\nu}\,d\sigma\,,
\end{align*}
where $X=(x_1,x_2)$ is the position vector and $\nu$ is the exterior normal to $\pa\Omega$. Recalling that $\operatorname{div}X=2$ and using the divergence theorem, we deduce
\begin{equation}
\label{eq:poho}
\int_{\Omega}4\,u\,d\mu\,=\,(4\,b+\beta^2)\,|E_o|\,-\,(4\,a+\alpha^2)\,|E_i|\,.
\end{equation}

From this formula we deduce some relations between  $a,b,\alpha,\beta$, summarized in the following result.
\begin{proposition}
\label{pro:acceptable}
Let $(\Omega,u)$ be a solution to problem~\eqref{eq:problem}.
\begin{itemize}
\item If $a>b$ and $b<u<a$ in $\Omega$, then $4a+\alpha^2>4b+\beta^2$, $\alpha\geq 0$, $\beta<0$,
\item If $a<b$ and $a<u<b$ in $\Omega$, then $4a+\alpha^2>4b+\beta^2$, $\alpha<0$, $\beta\geq 0$.
\end{itemize} 
\end{proposition}

\begin{proof}
We start by recalling that $\Omega=E_o\setminus\overline{E_i}$, so that in particular $|E_o|=|\Omega|+|E_i|$. As a consequence, we can rewrite~\eqref{eq:poho} in the following two forms
\begin{align*}
\int_{\Omega}\left(4\,u-4\,a-\alpha^2\right)\,d\mu\,&=\,(4\,b+\beta^2-4\,a-\alpha^2)\,|E_o|\,,
\\
\int_{\Omega}\left(4\,u-4\,b-\beta^2\right)\,d\mu\,&=\,(4\,b+\beta^2-4\,a-\alpha^2)\,|E_i|\,.
\end{align*}
If $b<u<a$ in $\Omega$, then the left hand side of the first equation is negative, whereas if $a<u<b$ in $\Omega$, then the left hand side of the second equation is negative. In both cases, we deduce $4a+\alpha^2>4b+\beta^2$.

The signs for $\alpha$ and $\beta$ follow immediately from the hypothesis that $a$ is the maximum (resp. minimum) of $u$ and $b$ is the minimum (resp. maximum) of $u$. The fact that $\beta\neq 0$ when $b<u<a$ and that $\alpha\neq 0$ when $a<u<b$ is a consequence of the Hopf Lemma (recall $\Delta u=-2<0$).
\end{proof}

In order to start our comparison argument, we first need to take a close look at the model solutions. 
In particular, we first have to make sure that we always have a model solution to compare with, for any acceptable value of $a,b,\alpha,\beta$.

 It can be easily checked that the rotationally symmetric solutions to $\De u=-2$ have the form
\begin{equation}
\label{eq:models}
u\,=\,\A-\frac{1}{2}|x|^2+\B\log |x|\,,
\end{equation}
with $\A,\B\in\R$. When $\B\leq 0$, these solutions are monotonically decreasing with respect to $|x|$. When $\B> 0$, these solutions are monotonically increasing for $|x|<\sqrt{\B}$ and they are monotonically decreasing for $|x|>\sqrt{\B}$. 

\begin{remark}
As in~\cite{ABM}, one can fix the value of $\A$ by means of a rescaling of the function and of the domain, ending up with the one-parameter family
$$
u\,=\,\frac{1-|x|^2}{2}+\B\log |x|\,.
$$
However, for our purposes in this paper this is not necessary, as it is actually easier to work directly with~\eqref{eq:models}.
\end{remark}

We expect that for any acceptable choice (according to Proposition~\ref{pro:acceptable}) of $a,b,\alpha,\beta$, there should be a solution of the form~\eqref{eq:models} solving our problem~\eqref{eq:problem}. This is true in the case $a<b$, as guaranteed by the following lemma:

\begin{lemma}
\label{le:comparison1}
If $a<b$, $\alpha<0$, $\beta\geq 0$, $4a+\alpha^2>4b+\beta^2$, then there exist constants $\A$ and $\B>0$ and two radii $0< r_i<r_o<\sqrt{\B}$ so that the function $u\,=\,\A-|x|^2/2+\B\log |x|$ satisfies 
problem~\eqref{eq:problem} in the annulus $\Omega=\{r_i<|x|<r_o\}$.
\end{lemma}

\begin{proof}
The proof of the lemma is more of a long exercise. We start by imposing the conditions we want on $\A,\B,r_o,r_i$, namely

\begin{equation}
\label{eq:system}
\begin{dcases}
\A-r_i^2/2+\B\log r_i\,=\,a
\\
\A-r_o^2/2+\B\log r_o\,=\,b
\\
\B-r_i^2\,=\,-\alpha\, r_i
\\
\B-r_o^2\,=\,\beta\, r_o
\end{dcases}
\end{equation}
and we employ the last two equations to obtain 
\begin{equation}
\label{eq:ro_ri}
r_i\,=\,\frac{\alpha\pm\sqrt{\alpha^2+4\B}}{2}\,,\qquad r_o\,=\,\frac{-\beta\pm\sqrt{\beta^2+4\B}}{2}\,.
\end{equation}
Since we are assuming $\alpha< 0$ and $\beta\geq 0$, in order for $r_o,r_i$ to be positive, we need to choose the plus sign in~\eqref{eq:ro_ri} and we need $\B>0$.

Substituting in the first two equations, we then deduce that we have a solution if and only if we can find a zero for the function
$$
F(\B)\,=\,4a+\alpha^2-4b-\beta^2+\alpha\sqrt{\alpha^2+4\B}+\beta\sqrt{\beta^2+4\B}+4\B\log\left(\frac{-\beta+\sqrt{\beta^2+4\B}}{\alpha+\sqrt{\alpha^2+4\B}}\right)\,.
$$
Taking the limit at zero and infinity, we compute
\begin{align*}
\lim_{\B\to 0}F(\B)\,&=\,4a-4b<0\,,
\\
\lim_{\B\to+\infty}F(\B)\,&=\,4a+\alpha^2-4b-\beta^2>0\,.
\end{align*}
It follows by continuity of $F$ that there is a value $0<\B<\infty$ such that $F(\B)=0$, as wished.
\end{proof}

The case $a>b$ appears to be more complicated. In the next lemma we will show that we can still find a model to compare with as long as we assume the additional condition $2a+\alpha^2\leq 2b+\beta^2$.
This is the first instance in which we notice complications in the case $a>b$. More specifically, it seems that negative values of $\B$ are harder to deal with. In this respect, the hypothesis $2a+\alpha^2\leq 2b+\beta^2$ is helpful as it forces $\B\geq 0$.

\begin{lemma}
\label{le:comparison2}
Let $a>b$, $\alpha\geq 0$, $\beta<0$, $4a+\alpha^2>4b+\beta^2$. If we further have $2a+\alpha^2\leq 2b+\beta^2$, then there exist constants $\A$ and $\B\geq 0$ and two radii $\sqrt{M}\leq r_i<r_o$ so that the function $u\,=\,\A-|x|^2/2+\B\log |x|$ satisfies 
problem~\eqref{eq:problem} in the annulus $\Omega=\{r_i<|x|<r_o\}$. 
%\begin{comm}The radius $r_i$ is equal to zero if and only if $\alpha=0$ and $4a=4b+2\beta^2$, in which case $\Omega$ is actually a ball of radius $r_o$. Furthermore, $\B\geq 0$ if $4a+2\alpha^2\leq 4b+2\beta^2$ and $\B<0$ otherwise.\end{comm}
\end{lemma}

\begin{proof}
%Proceedining as in the proof of Lemma~\ref{le:comparison1}, we reduce ourselves to finding a zero for the function
%$$
%F(\B)\,=\,4a+\alpha^2-4b-\beta^2\pm\alpha\sqrt{\alpha^2+4\B}\pm\beta\sqrt{\beta^2+4\B}+4\B\log\left(\frac{-\beta\pm\sqrt{\beta^2+4\B}}{\alpha\pm\sqrt{\alpha^2+4\B}}\right)\,,
%$$
%were the signs $\pm$ depend on the choice of the sign of $r_o$ and $r_i$ in~\eqref{eq:ro_ri}.
Proceeding as in the proof of Lemma~\ref{le:comparison1}, we find again the expressions~\eqref{eq:ro_ri} for $r_o$ and $r_i$. Notice that, in contrast with the previous case, both signs for $\pm$ in~\eqref{eq:ro_ri} are possible. However, since we are ultimately interested in having $\B\geq 0$, this again essentially forces us to choose  the plus sign in the formul\ae~\eqref{eq:ro_ri} for $r_i$ and $r_o$.
Therefore, as in the previous lemma, we reduce ourselves to finding a zero for the function
$$
F(\B)\,=\,4a+\alpha^2-4b-\beta^2+\alpha\sqrt{\alpha^2+4\B}+\beta\sqrt{\beta^2+4\B}+4\B\log\left(\frac{-\beta+\sqrt{\beta^2+4\B}}{\alpha+\sqrt{\alpha^2+4\B}}\right)\,.
$$
%We now move our attention to the first case, namely $a>b$, $\alpha\geq 0$, $\beta< 0$, $4a+\alpha^2>4b+\beta^2$.
%Since we esssentially want to focus on the case $\B\geq 0$, we are essentially forced to choose the plus sign in the formul\ae~\eqref{eq:ro_ri} for $r_i$ and $r_o$.
%(\begin{comm}when $\B=0$, one could in principle choose the minus sign in the formula for $r_i$, giving $r_i=0$, but the model solution with $\B=0$ and $r_i=0$ satisfies $\alpha=0$, which means that plus or minus sign gives the same result\end{comm}). 
Then we compute
\begin{align*}
F(0)\,&=\,4a+2\alpha^2-4b-2\beta^2\,,
\\
\lim_{\B\to+\infty}F(\B)\,&=\,4a+\alpha^2-4b-\beta^2>0\,.
\end{align*}
Therefore, since we are assuming $4a+2\alpha^2-4b-2\beta^2\leq 0$, we conclude again that there is a value $0\leq\B<\infty$ such that $F(\B)=0$.
%\begin{comm}
%If instead $\B< 0$, then we can choose any sign in the formul\ae~\eqref{eq:ro_ri}. This complicates the analysis, which is anyway not particularly interesting and more of a long exercise. Since this case will be of less interest for us anyway, we leave the details to the interested reader.
%\end{comm}
\end{proof}

\section{Exploiting the comparison argument}
\label{sec:exploit}

In the previous section we have developed a way to compare a general solution of~\eqref{eq:problem} with a model solution~\eqref{eq:models}. Building on that, in this section we develop our comparison technique. In Subsection~\ref{sub:grad_est} we prove a crucial gradient estimate, that is then exploited in Subsection~\ref{sub:area_bounds} to prove area bounds for the boundary components $\Gamma_i$ and $\Gamma_o$. The results in both these subsection work in both the case $a<b$ and $b<a$.  In Subsection~\ref{sub:proof_main} we will then specialize to the case $a<b$ and we will prove Theorem~\ref{thm:main}.

\subsection{Gradient estimate}
\label{sub:grad_est}

Let $u$ be a solution to problem~\eqref{eq:problem} for some fixed values of $a,\alpha,b,\beta$. In the case $a>b$, assume further that $2a+\alpha^2\leq 2b+\beta^2$. Then, Lemmata~\ref{le:comparison1} and~\ref{le:comparison2} tell us that there exist constants $\A\in\R$, $\B\geq 0$, $r_o>0$, $r_i>0$ such that the model solution
$$
u\,=\,\A-\frac{1}{2}|x|^2+\B\log|x|\,,
$$
in the annulus $\{r_i<|x|<r_o\}$, solves problem~\eqref{eq:problem} for the same values of  $a,\alpha,b,\beta$. 

From now on, the constants $\A\in\R$, $\B\geq 0$, $r_o>0$, $r_i>0$ will always be the ones prescribed by Lemmata~\ref{le:comparison1} and~\ref{le:comparison2}.
We are now ready to set up our comparison argument, in the spirit of~\cite{ABM}. 
We start by defining the {\em pseudo-radial function} $\Psi:\Omega\to \R$ implicitly via the following identity
\begin{equation}
\label{eq:pr}
u\,=\,\A-\frac{1}{2}\Psi^2+\B\log\Psi\,.
\end{equation} 
From the Implicit Function Theorem, it is easily seen that $\Psi$ is well defined as long as $\Psi^2\neq\B$ inside $\Omega$. From now on we will always assume that $a<u<b$ or $b<u<a$, and from this it easily follows that the function $\Psi$ actually takes values between $r_i$ (attained at $\Gamma_i$) and $r_o$ (attained at $\Gamma_o$). On the other hand, Lemmata~\ref{le:comparison1} and~\ref{le:comparison2} tell us that either $r_i<r_o\leq \sqrt{\B}$ or $\sqrt{\B}\leq r_i<r_o$. Therefore, for all our purposes, the function $\Psi:\Omega\to[r_i,r_o]$ will always be well defined.
%\begin{lemma}
%Let $u$ be a solution to problem~\eqref{eq:problem} with respect to some choice of $a,\alpha,b,\beta$. Suppose that $a<u<b$ or $b<u<a$.
%Then the function $\Psi:\Omega\to[r_i,r_o]$ obtained implicitly by~\eqref{eq:pr} is well defined.
%\end{lemma}

%It is clear that, if $u$ is rotationally symmetric, then $\Psi$ coincides with the radial coordinate $|x|$.
%Furthermore, notice that, by construction, $\Psi=r_i$ on $\Gamma_i$ and $\Psi=r_o$ on $\Gamma_o$.  

We also introduce the two functions 
\begin{equation}
\label{eq:W0}
W\,=\,|\D u|^2\quad \hbox{and}\quad W_0\,=\,\left(\frac{\B-\Psi^2}{\Psi}\right)^2\,.
\end{equation}
The function $W_0$ corresponds to the value that $|\D u|^2$ would have on the model solution. In fact, it is easily checked that, for the model solutions~\eqref{eq:models}, it holds $W_0\equiv W$ pointwise. Our strategy is then that of comparing $W$ and $W_0$ in $\Omega$, ultimately trying to prove that they have to coincide.

First of all, notice that, by construction, $W$ and $W_0$ coincide on the boundary of $\Omega$, namely when $\Psi=r_o$ and when $\Psi=r_i$. We will now show that, under our hypotheses, $W$ is actually controlled by $W_0$ on the whole $\Omega$.

\begin{theorem}[Gradient Estimates]
\label{thm:gradient_estimate}
Let $(\Omega,u)$ be a solution to problem~\eqref{eq:problem} with respect to some choice of $a,\alpha,b,\beta$. Suppose that $a<u<b$ or $b<u<a$. In the case $b<u<a$, assume further that $2a+\alpha^2\leq 2b+\beta^2$.
Let $W$, $W_0$ be defined as above. Then it holds 
\begin{equation}
\label{eq:gradient_estimate}
W \, \leq \,  W_0 \quad\hbox{in }\Omega\,,
\end{equation}
Moreover, if $W=W_0$ at some point in the interior of $\Omega$, then $(\Omega,u)$ is rotationally symmetric.
\end{theorem}

\begin{remark}
This theorem is clearly inspired by~\cite[Theorem~3.1]{ABM}, however we mention that similar gradient estimates have already found applications in a number of problems. One of the most notable ones is the paper~\cite{Beig_Simon} and the subsequent developments in~\cite{Chrusciel_Simon,Lee_Neves}, where a gradient comparison argument very much resembling Theorem~\ref{thm:gradient_estimate} has been obtained and exploited for static spacetimes in General Relativity. 
The introduction of the pseudo-radial function is instead more recent (it was exploited in the series of papers~\cite{BCM,Borghini_Mazzieri_I,
Borghini_Mazzieri_II} for static spacetimes). This function is really helpful as it allows to have an explicit formula for $W_0$. This is crucial in the computations that follows: we will see that $\Psi$ will appear in all the important computations in this section.
\end{remark}

\begin{proof}
The proof is essentially the same as the one in~\cite{ABM}, but let us give some comments. Let $\A\in\R$, $\B\geq 0$, $r_o>0$, $r_i>0$ be the constants prescribed by Lemmata~\ref{le:comparison1} and~\ref{le:comparison2}.
%We start by using the Bochner formula to compute the following:
%\begin{align}
%\notag
%\De (W-W_R)\,&=\,\De|\na u|^2\,+\,2\left(1+\frac{R^2}{\Psi^2}\right)\De u\,-\,\frac{4\,R^2}{\Psi^2(R^2-\Psi^2)}|\na u|^2
%\\
%\label{eq:De_WWo}
%&=\,2\,|\nana u|^2\,-\,4\left(1+\frac{R^2}{\Psi^2}\right)\,-\,\frac{4\,R^2}{\Psi^2(R^2-\Psi^2)}|\na u|^2
%\end{align}
The crucial estimate is deduced starting from the following quantity
$$
\DD u\,+\,\frac{2\B}{(\B -\Psi^2)^2}du\otimes du\,+\left[1\,-\,\frac{\B}{(\B -\Psi^2)^2}|\D u|^2\right]g_{\R^2}\,.
$$
The heuristic behind the choice of the quantity above is essentially the fact that it can be computed to be zero on the model solution. 
Computing explicitly its square norm and using the fact that it is nonnegative, we obtain the following estimate for the hessian of $u$:
\begin{equation*}
|\nana u|^2\,\geq\,-\,\frac{2\,\B}{(\B-\Psi^2)^2}\langle\na(W-W_0)|\na u\rangle\,
\\
+\,2\,\left[1\,+\,\frac{2\,\B^2}{\Psi^2(\B-\Psi^2)^2}|\D u|^2\,-\,\frac{\B^2}{(\B-\Psi^2)^4}|\na u|^4\right]\,.
\end{equation*}
Starting now from the Bochner formula and plugging in this estimate, we obtain an elliptic inequality for the quantity $W-W_0$
\begin{multline*}
\De(W-W_0)\,\geq\,-\,\frac{4\,\B}{(\B-\Psi^2)^2}\langle\na(W-W_0)|\na u\rangle\,
\\
+\,\frac{4\,\B}{(\B-\Psi^2)^2}\,\left[1\,-\,\frac{\B}{(\B-\Psi^2)^2}|\D u|^2\right](W-W_0)\,,
\end{multline*}
and considering the function $F_\gamma=\gamma\,(W-W_0)$, where $\gamma=\gamma(\Psi)>0$, one gets
\begin{multline}
\label{eq:elliptic_ineq_F}
\De F_\gamma\,\geq\,-\,\frac{2\Psi}{\B-\Psi^2}\left[\frac{\gamma'}{\gamma}\,-\,\frac{2\,\B}{\Psi(\B-\Psi^2)}\right]\langle\na F_\gamma|\na u\rangle\,-\,\frac{2\,\Psi}{\B-\Psi^2}\,\left[\frac{\gamma'}{\gamma}\,-\,\frac{2\,\B}{\Psi(\B-\Psi^2)}\right]\,F_\gamma
\\
+\,\frac{W}{W_0}\,\left[\left(\frac{\gamma'}{\gamma}\right)'\,-\,\left(\frac{\gamma'}{\gamma}\right)^2\,+\,\frac{\Psi^2+5\B}{\Psi(\B-\Psi^2)}\frac{\gamma'}{\gamma}\,-\,\frac{4\,\B^2}{\Psi^2(\B-\Psi^2)^2}\right]F_\gamma\,,
\end{multline}
where we have  used $'$ to denote the differentiation with respect to $\Psi$.
We now need to find a function $\gamma$ such that the coefficients  of the zero order terms have the right sign.
A good choice is to set
$$
\frac{\gamma'}{\gamma}\,=\,\frac{2\B}{\Psi(\B-\Psi^2)}\,,
$$
which corresponds to choosing
$$
\gamma\,=\,\frac{\Psi}{\sqrt{W_0}}\,=\,\frac{\Psi^2}{|\B-\Psi^2|}\,.
$$
We know that $\Psi^2\neq \B$ inside $\Omega$, hence $\gamma$ is well defined in the interior of $\Omega$. However, $\Psi^2$ may attain the value $\B$ on the boundary $\pa\Omega$: this happens when $\alpha=0$ (which implies from~\eqref{eq:ro_ri} that $r_i=\sqrt{\B}$) or when $\beta=0$ (which implies from~\eqref{eq:ro_ri} that $r_o=\sqrt{\B}$). Let us for the moment assume that this does not happen, so that $F_\gamma$ is well defined on the whole $\Omega$.
With this choice of $\gamma$, it is easily seen that $F_\gamma$ satisfies
\begin{equation}
\label{eq:elliptic_ineq_F_simplified}
\De F_\gamma\,-\,\frac{8\,\B \, \Psi^2}{(\B-\Psi^2)^4}\,|\D u|^2\,F_\gamma\,\geq\,0\,.
\end{equation}
Since $\B\geq 0$, it follows that $F_\gamma$ satisfies the Maximum Principle in $\Omega$. Furthermore, $W = W_0$ on $\pa\Omega$ by construction, hence the Maximum Principle implies that $F_\gamma\leq 0$ (equivalently,
$W\leq W_0$) on the whole $\Omega$. Furthermore, if the equality $W=W_0$ holds at one point $p$ in the interior of $N$, then, applying the Strong Maximum Principle in a neighborhood of $p$, we deduce that $W=W_0$ on the whole $\Omega$. It is then easy to conclude the desired rigidity statement with standard arguments, see for instance~\cite[Theorem~4.2]{Borghini}. 

It remains to discuss the case where $\Psi^2=\B$ on one of the boundary components, say $\Gamma_i$ (for $\Gamma_o$, the same argument apply). In this case, we consider $\ep>0$ and a small neighborhood $U_\ep=\{x\in\Omega\,:\,d(x,\Gamma_i)<\ep\}$ of $\Gamma_i$, and we  work on the domain $\Omega_\ep=\Omega-U_\ep$. Clearly $F_\gamma$ is well defined and satisfies the Maximum Principle in $\Omega_\ep$. We will then obtain the desired result by taking $\ep\to 0$, provided we can prove that $F_\gamma\to 0$ as we approach $\Gamma_i$. It is easily seen that $F_\gamma$ goes to zero if and only if $W/\sqrt{W_0}$ goes to zero, which in turn is equivalent to $|\D u|^2/\sqrt{a-u}\to 0$, see~\cite[Lemma~A.1]{ABM}.
The latter is granted by the Reverse {\L}ojasiewicz Inequality~\cite[Theorem~2.2]{BCM} (to be more precise,~\cite[Theorem~2.2]{BCM} cannot be applied directly, as it is written for interior points only; however the proof extends without modifications to the case at hand).
\end{proof}

\subsection{Area bounds}
\label{sub:area_bounds}

Following again~\cite{ABM}, we can now exploit the gradient estimate proven above to obtain area bounds for both $\Gamma_i$ and $\Gamma_o$.

\begin{proposition}
\label{pro:area_bound_Gamma}
Let $(\Omega,u)$ be a solution to problem~\eqref{eq:problem} with respect to some choice of $a,\alpha,b,\beta$. Suppose that $a<u<b$ or $b<u<a$. In the case $b<u<a$, assume further that $2a+\alpha^2\leq 2b+\beta^2$. Then
\begin{equation}
|\Gamma_i|\,\leq\, 2\pi\,r_i\quad\hbox{ and }\quad
|\Gamma_o|\,\geq\, 2\pi\, r_o\,.
\end{equation}
\end{proposition}

\begin{proof}
We will prove the area bound for $\Gamma_i$ only, as the one for $\Gamma_o$ is proven in the same way (up to some sign differences coming from the fact that $\D\Psi$ is pointing outside $\Omega$ on $\Gamma_o$). Let us also suppose for the moment that $\beta\neq 0$, so that the gradient of $u$ does not vanish on $\Gamma_i$. Notice that 
$$
\D\Psi\,=\,\frac{\Psi}{\B-\Psi^2}\D u\,=\,\frac{\D u}{\sqrt{W_0}}
$$
has norm $|\D\Psi|=\sqrt{W/W_0}$ equal to $1$ on $\Gamma_i$ and less than $1$ inside $\Omega$ thanks to~\eqref{eq:gradient_estimate}. In particular, the inequality $\langle\D|\D\Psi|^2\,|\,\D\Psi\rangle\leq 0$ holds on $\Gamma_i$ and the curvature $\kappa$ of $\Gamma_i$ at any point of $\Gamma_i$ satisfies
\begin{equation}
\label{eq:kappa}
\kappa\,=\,\frac{\De\Psi\,-\,\DD\Psi(\D\Psi,\D\Psi)}{|\D\Psi|}
\,=\,\De\Psi\,-\,\langle\D|\D\Psi|^2\,|\,\D\Psi\rangle\geq\De\Psi\,=\,1/\Psi=1/r_i\,.
\end{equation}
Integrating $\kappa\geq r_i$ on $\Gamma_i$, recalling that the total curvature of a simple curve is equal to $2\pi$, we get
$$
2\pi\,\geq |\Gamma_i|/r_i\,.
$$
This concludes the proof of the area bound for $\Gamma_i$. %The same strategy works to prove the area bound for $\Gamma_o$, with some sign difference coming from the fact that $\D\Psi$ is pointing outside $\Omega$ on $\Gamma_o$.

Let us now discuss the case $\beta=0$. In this case, one has $W=W_0=0$ on $\Gamma_i$, so $|\D\Psi|=\sqrt{W/W_0}$ is not well defined. Let $p\in\Gamma_i$ and let us extend $u$ to a smooth function in a neighborhood $U$ of $p$. We can do it by means of Whitney Extension Theorem~\cite{Whitney}, that clearly applies to our case since $u$ is smooth up to the boundary (see for instance~\cite[Theorem~6.19]{Gilbarg_Trudinger}). With a slight abuse of notation, let us still call $u$ the smooth function defined on $U$ that coincides with the original $u$ in $\overline{\Omega}\cap U$.
In the spirit of~\cite[Theorem~3.1]{BCM}, we consider the signed distance function $r$ to $\Gamma_i$. If $U$ is taken to be small enough, the function $r$ is known to be smooth and can be used as a coordinate. We can then proceed exactly as in~\cite[Theorem~3.1]{BCM} to show that 
\begin{equation}
\label{eq:expansion_u}
u\,=\,a-r^2+o(r^2)  
\end{equation}
inside $U$. From this, it easily follows that $W=4r^2+o(r^2)$ and $W_0=4r^2+o(r^2)$, hence $|\D\Psi|$ goes to $1$ as we approach $\Gamma_i$. We can then mimik the argument given for the case  $\beta\neq 0$: the first equality in~\eqref{eq:kappa} still holds for the curvature of the level sets of $u$ close to $\Gamma_i$, therefore at the limit we obtain again that the curvature of $\Gamma_i$ is greater than or equal to $1/r_i$, from which one concludes.
%\begin{comm}
%In the case where $\alpha=0$ (resp. $\beta=0$), one needs to be slightly more careful as $\D\Psi$ is ill-defined on $\Gamma_i$ (resp., $\Gamma_o$). However, $\D\Psi$ is well defined on the whole $\Omega$, so one can repeat the same procedure by looking at the limit.
%\end{comm}
\end{proof}

\subsection{Proof of Theorem~\ref{thm:main}}
\label{sub:proof_main}

In this subsection, we focus on the case $a<b$ and we prove Theorem~\ref{thm:main}. For simplicity, let us start by discussing the case where $\D u$ does not vanish on the boundary of $\Omega$, namely for the moment let us assume $\beta\neq 0$ (notice that we know that $\alpha$ is strictly positive in the case $a<b$, so $\alpha$ cannot vanish).
Using the divergence theorem and recalling that $\De u=-2$, we compute
\begin{align}
\notag
\int_{\Omega}\frac{2}{\B-\Psi^2}\,\rmd\mu\,=\,\int_{\Omega} -\frac{\De u}{\B-\Psi^2}\,\rmd\mu\,&=\,\int_{\Omega}\left\langle\D\left(\frac{1}{\B-\Psi^2}\right)\,\bigg|\,\D u\right\rangle\,\rmd\mu\,-\,\int_{\pa \Omega}\frac{\langle \D u\,|\,\nu\rangle}{\B-\Psi^2}\,\rmd\sigma
\\
\label{eq:divegence_theorem}
&=\,\int_{\Omega}\frac{2\,\Psi^2}{(\B-\Psi^2)^3}\,|\D u|^2\,\rmd\mu\,-\,\int_{\pa \Omega}\frac{\langle \D u\,|\,\nu\rangle}{\B-\Psi^2}\,\rmd\sigma\,,
\end{align}
where $\nu$ is the outward unit normal to $\pa\Omega=\Gamma_i\sqcup\Gamma_o$.
Since we are in the case $a<b$, then  $\nu=-\D u/|\D u|$ on $\Gamma_i$ and $\nu=\D u/|\D u|$ on $\Gamma_o$. Setting $W=|\D u|^2$ as usual, and recalling the expression~\eqref{eq:W0} of $W_0$ in terms of $\Psi$, the identity above can be written as
\begin{equation}
\label{eq:crucial_mon}
\int_{N_\ep}\frac{2\,\Psi^2}{(\B-\Psi^2)^3}\,(W_0-W)\,\rmd\mu\,=\,\int_{\Gamma_i}\frac{|\D u|}{\B-\Psi^2}\,\rmd\sigma\,-\,\int_{\Gamma_o}\frac{|\D u|}{\B-\Psi^2}\,\rmd\sigma\,.
\end{equation}

Notice that the left hand side is nonnegative because of~\eqref{eq:gradient_estimate}. Furthermore, since we are focusing on the case $a<b$, we have $\Psi<r_o<\sqrt{\B}$ hence
$$
\frac{|\D u|}{\B-\Psi^2}\,=\,\frac{1}{\Psi}\sqrt{\frac{W}{W_0}}\,.
$$
Since by construction $W=W_0$ on $\pa\Omega$, from~\eqref{eq:crucial_mon} we immediately deduce 
\begin{equation}
\label{eq:converse_area_bound}
0\leq \frac{|\Gamma_i|}{r_i}-\frac{|\Gamma_o|}{r_o}\,.
\end{equation}
On the other hand, the area bounds in Proposition~\ref{pro:area_bound_Gamma} give us the opposite inequality.
Therefore the equality must hold in~\eqref{eq:converse_area_bound}, and from~\eqref{eq:crucial_mon} in particular we obtain that $W\equiv W_0$ in $\Omega$. The result now follows from the rigidity statement in Theorem~\ref{thm:gradient_estimate}.

It remains to discuss the case where $\beta$ is equal to zero. In this case, the boundary term in~\eqref{eq:divegence_theorem} is ill-defined, as one has $\Psi^2=r_o^2=\B$ on $\Gamma_o$. To avoid this problem, we just need to apply the divergence theorem in $\Omega\setminus\{u>b-\ep\}$, and then take the limit as $\ep\to 0$. From~\eqref{eq:expansion_u}
% and Lemma~\ref{le:W_Wo} 
it follows easily that the quantity $\langle \D u\,|\,\nu\rangle/(\B-\Psi^2)$ goes to $1/\Psi=1/r_o$ as we approach $\Gamma_o$. The proof then proceeds exactly as in the case $\beta\neq 0$.

\section{Refined Pohozaev}
\label{sec:poho}

In this section we focus on the case $a>b$ and we prove Theorem~\ref{thm:reichel} under the additional hypothesis that $2a+\alpha^2\leq 2b+\beta^2$. Let us give the precise statement for the convenience of the reader.

\begin{theorem}
\label{thm:main2}
Let $(\Omega,u)$ be a solution to problem~\eqref{eq:problem} with $E_i\neq\emptyset$. Suppose further that $a>b$,  $b<u<a$ in $\Omega$ and $2a+\alpha^2\leq 2b+\beta^2$. Then $\Omega$ is an annulus and $u$ is rotationally symmetric with $\pa u/\pa |x|<0$.
\end{theorem}

As discussed, all the arguments in Subsections~\ref{sub:grad_est} and~\ref{sub:area_bounds} still work in this case.
Unfortunately however, the strategy employed in Subsection~\ref{sub:proof_main} to complete the proof for $a<b$, does not work when $a>b$. In fact, one can check that in this case the inequality that one gets at the end has the opposite sign of~\eqref{eq:converse_area_bound}, hence it does not combine with the area bounds from Proposition~\ref{pro:area_bound_Gamma} anymore.
We will then need to exploit a more delicate argument based on sharp estimates coming from the Pohozaev identity and the isoperimetric inequality.

Let us start by writing
$$
4\,u\,=\,2\ffi+(4u-2\ffi)\,=\,-\ffi\De u+(4u-2\ffi)\,,
$$
where $\ffi$ is a function of $u$ to be chosen later. We then use the divergence theorem to write the following sequence of identities (we denote by $\dot\ffi$ the derivative of $\ffi$ with respect to $u$)
\begin{align}
\notag
\int_{\Omega}4u\, d\mu\,&=\,\int_{\Omega}[-\ffi\De u+(4u-2\ffi)]d\mu
\\
\notag
&=\,\int_{\Omega}[\dot\ffi |\D u|^2+4u-2\ffi]d\mu\,-\,\int_{\pa\Omega}\ffi\frac{\pa u}{\pa\nu}d\sigma
\\
\label{eq:aux_integral}
&=\,\int_{\Omega}\dot\ffi\left(W-\frac{2\ffi-4u}{\dot\ffi}\right)d\mu\,-\,\alpha \ffi(a)|\Gamma_i|
\,-\,\beta \ffi(b)|\Gamma_o|\,.
\end{align}
If we choose
\begin{equation}
\label{eq:ffi}
\ffi\,=\,2u-\frac{\Psi^4-4\B\Psi^2+4\B^2\log\Psi+k}{2(\B-\Psi^2)}\,,
\end{equation}
where $k\in\R$ is a constant,
then one can check that it holds $(2\ffi-4u)/\dot\ffi=W_0$. Combining~\eqref{eq:aux_integral} with the Pohozaev identity~\eqref{eq:poho}, with this choice of $\ffi$ we obtain
\begin{equation*}
\int_{\Omega}\dot\ffi\left(W-W_0\right)d\mu\,=\,(4\,b+\beta^2)\,|E_o|\,-\,(4\,a+\alpha^2)\,|E_i|\,+\,\alpha \ffi(a)|\Gamma_i|
\,+\,\beta \ffi(b)|\Gamma_o|\,.
\end{equation*}
Since $|\Omega|=|E_o|-|E_i|$, we can write the above equation in the following way:
\begin{equation*}
\int_{\Omega}\dot\ffi\left(W-W_0\right)d\mu\,=\,(4\,b+\beta^2-4\,a-\alpha^2)\,|E_o|\,+\,(4\,a+\alpha^2)\,|\Omega|\,+\,\alpha \ffi(a)|\Gamma_i|
\,+\,\beta \ffi(b)|\Gamma_o|\,.
\end{equation*}
Recall from Proposition~\ref{pro:acceptable} that $4a+\alpha^2>4b+\beta^2$. We can then combine the isoperimetric inequality, the area bounds in Proposition~\ref{pro:area_bound_Gamma} and the identity 
$$
-2\,|\Omega|\,=\,\int_{\Omega}\De u\, d\mu\,=\,\int_{\pa\Omega}\frac{\pa u}{\pa\nu}d\sigma\,=\,\alpha\,|\Gamma_i|+\beta\,|\Gamma_o|
$$
to obtain the following inequality
\begin{multline*}
\int_{\Omega}\dot\ffi\left(W-W_0\right)d\mu\,\geq
\\
\left[\frac{4\,b+\beta^2-4\,a-\alpha^2}{2}r_o+\beta \ffi(b)-2a\beta-\frac{\alpha^2}{2}\beta \right]\,|\Gamma_o|\,+\,\left[\ffi(a)\alpha-2a\alpha-\frac{\alpha^3}{2}\right]|\Gamma_i|\,.
%\\
%\int_{\Omega}\dot\ffi\left(W-W_0\right)d\mu\,&\geq\,\left[\frac{4\,b+\beta^2-4\,a-\alpha^2}{2}r_i+\alpha \ffi(a)-2b\alpha-\frac{\beta^2}{2}\alpha \right]\,|\Gamma_i|\,+\,\left[\ffi(b)\beta-2b\beta-\frac{\beta^3}{2}\right]|\Gamma_o|\,.
\end{multline*}
Recall that the constants appearing in the above formula are related as follows:
\begin{equation*}
a\,=\,\A-\frac{r_i^2}{2}+\B\log r_i\,,\quad b\,=\,\A-\frac{r_o^2}{2}+\B\log r_o\,,\quad \alpha\,=\,r_i-\frac{\B}{r_i}\,,\quad \beta\,=\,\frac{\B}{r_o}-r_o\,.
\end{equation*}
Using these identities and formula~\eqref{eq:ffi} for $\ffi$, with some computations we can rewrite the above inequality as follows:
\begin{equation}
\label{eq:case1}
\int_{\Omega}\dot\ffi\left(W-W_0\right)d\mu\,\geq\,\frac{\B(4a+\alpha^2-4\A-\B)+k}{2}\,\left(\frac{|\Gamma_i|}{r_i}\,-\,\frac{|\Gamma_o|}{r_o}\right)\,.
%\\
%\label{eq:case2}
%\int_{\Omega}\dot\ffi\left(W-W_0\right)d\mu\,&\geq\,\frac{\B(4b+\beta^2-4\A-\B)+k}{2}\,\left(\frac{|\Gamma_i|}{r_i}\,-\,\frac{|\Gamma_o|}{r_o}\right)\,.
\end{equation}
Let us now analyze $\dot\ffi$. Differentiating~\eqref{eq:ffi}, we get
$$
\dot\ffi\,=\,\frac{\Psi^2}{(\B-\Psi^2)^3}\left[4\B\Psi^2-\Psi^4-4\B^2\log\Psi-k\right]\,.
$$
One can check that the quantity in square brakets is monotonically decreasing in $\Psi$. Since we know from Lemma~\ref{le:comparison2} that $\Psi^2\geq r_i^2\geq \B$, in order for $\dot\ffi$ to be positive in $\Omega$ it is then sufficient to choose
$$
k\,\geq\,4\B r_i^2-r_i^4-4\B^2\log r_i\,=\,4\A\B+\B^2-4a\B-\alpha^2 r_i^2
$$
Choosing $k$ exactly equal to the above value, recalling from~\eqref{eq:gradient_estimate} that $W\leq W_0$, we then get from~\eqref{eq:case1}:
$$
0\,\geq\,\frac{\alpha^2(\B-r_i^2)}{2}\,\left(\frac{|\Gamma_i|}{r_i}\,-\,\frac{|\Gamma_o|}{r_o}\right)
$$
Since $r_i^2\geq\B$, this implies
$$
\frac{|\Gamma_i|}{r_i}\,-\,\frac{|\Gamma_o|}{r_o}\geq 0
$$
But we have the opposite inequality from Proposition~\ref{pro:area_bound_Gamma}, therefore everything must be an equality. As a consequence, $W=W_0$ on the whole $\Omega$ and we conclude using the rigidity statement in Theorem~\ref{thm:gradient_estimate}.

\section{Further comments}
\label{sec:comments}

In the previous sections we have shown that we are able to deal with the case $a>b$ and prove Theorem~\ref{thm:reichel} if we assume the additional hypothesis $2a+\alpha^2\leq 2b+\beta^2$. While this hypothesis is somewhat restrictive, it still allows to get some nontrivial applications. As a particularly relevant example, we now argue that our weaker version of  Theorem~\ref{thm:reichel}, namely Theorem~\ref{thm:main2}, is still strong enough to deal with the cases of interest in~\cite{ABM}. There, Theorem~\ref{thm:reichel} was invoked in the proof of~\cite[Theorem~B]{ABM} on a domain $\Omega_o$ (the outer domain, in the terminology of~\cite{ABM}) for a function $u$ satisfying problem~\eqref{eq:problem} with $\alpha=0$, $b=0$ and $\beta^2/2a\geq 1$ (the latter inequality followed from~\cite[Theorem~2.1]{ABM}). Under those hypotheses, the inequality $2a+\alpha^2\leq 2b+\beta^2$ is trivially satisfied. In other words, Theorem~\ref{thm:main2} is enough for the intended applications in~\cite{ABM}.

%Nevertheless, it is clear that in the case $a>b$ our method is less effective than in the case $a<b$. In this section, we discuss in more details the complications that one encounters when $a>b$, the purpose being that of understanding what can go wrong with our method and how one may try to overcome  complications. The rest of this section is therefore devoted to further comment on the hypothesis $4a+2\alpha^2\leq 4b+2\beta^2$. Specifically, we discuss what can go wrong when this hypothesis is not in place and what are the complications that one would need to overcome in order to prove Theorem~\ref{thm:reichel} in full generality using our approach. 

The rest of this section is devoted to further comment on the hypothesis $2a+\alpha^2\leq 2b+\beta^2$. Specifically, we discuss what can go wrong when this hypothesis is not in place and what are the complications that one would need to overcome in order to prove Theorem~\ref{thm:reichel} in full generality using our approach. 

The first basic step that one needs in order to start our comparison argument is of course to have a model to compare with. In other words, given a solution $(\Omega,u)$ of problem~\eqref{eq:problem} for some $a,b,\alpha,\beta$ with $a>b$, we need to have a model solution~\eqref{eq:models} solving the same problem. This is granted by Lemma~\ref{le:comparison2} when $2a+\alpha^2\leq 2b+\beta^2$. Unfortunately, an analogous result seems harder to prove when $2a+\alpha^2> 2b+\beta^2$.

In case we are able to show that there is a model solution to compare with, there is still another crucial complication, namely the proof of Theorem~\ref{thm:gradient_estimate}. In fact, in order to prove the gradient estimate $W\leq W_0$, we relied on the Maximum Principle applied to the elliptic inequality~\eqref{eq:elliptic_ineq_F_simplified}. Unfortunately, when $\B<0$ (this can happen when $2a+\alpha^2> 2b+\beta^2$), the zeroth order term of~\eqref{eq:elliptic_ineq_F_simplified} has the wrong sign, hence we cannot apply the Maximum Principle. It is possible with some work to find a workaround, at least under the additional hypothesis that $r_i^2\geq-\B$.

\begin{theorem}[Gradient Estimates, $\B< 0$]
\label{thm:gradient_estimate_2}
Let $(\Om,u)$ be a solution to problem~\eqref{eq:problem} and suppose that $b<u<a$. Suppose that there exist $\A,\B,r_o,r_i$ such that the corresponding model solution~\eqref{eq:models} solves the problem for the same $a,b,\alpha,\beta$. Let $W$, $W_0$ be defined as in~\eqref{eq:W0}. If $r^2_i\geq -\B>0$,
then it holds 
\begin{equation}
\label{eq:gradient_estimate_2}
W \, \leq \,  W_0 \quad\hbox{in }\Omega\,,
\end{equation}
Moreover, if $W=W_0$ at some point in the interior of $\Omega$, then $(\Omega,u)$ is rotationally symmetric.
\end{theorem}

\begin{proof}
We start by observing that we can write~\eqref{eq:elliptic_ineq_F} as
\begin{multline}
\label{eq:elliptic_ineq_F_version2}
\De F_\gamma\,\geq\,\frac{2\Psi}{\Psi^2-\B}\left[\frac{\gamma'}{\gamma}\,+\,\frac{2\,\B}{\Psi(\Psi^2-\B)}\right]\langle\na F_\gamma|\na u\rangle\,-\,\frac{2\,\Psi}{\Psi^2-\B}\,\left[\frac{\gamma'}{\gamma}\,+\,\frac{2\,\B}{\Psi(\Psi^2-\B)}\right]\,\frac{F_\gamma^2}{\gamma W_0}
\\
+\,\frac{W}{W_0}\,\left[\left(\frac{\gamma'}{\gamma}\right)'\,-\,\left(\frac{\gamma'}{\gamma}\right)^2\,+\,\frac{\Psi^2-5\B}{\Psi(\Psi^2-\B)}\frac{\gamma'}{\gamma}\,+\,\frac{4\,\B}{\Psi^2(\Psi^2-\B)}\right]F_\gamma\,.
\end{multline}

If we ask the quantity in the latter square bracket to be equal to zero, we obtain the following formula for $\gamma$
$$
\gamma\,=\,\frac{\Psi^2}{\left|2\Psi^2-(\Psi^2+\B)(\log\Psi- k)\right|}\,.
$$
where $k\in\R$ is a constant.
In particular we get
\begin{equation}
\label{eq:first_square_bracket_term}
\frac{\gamma'}{\gamma}\,+\,\frac{2\,\B}{\Psi(\Psi^2-\B)}\,=\,\frac{\Psi^4+4\B \Psi^2-\B^2-4\B\Psi^2(\log\Psi- k)}{\Psi(\Psi^2-\B)[2\Psi^2-(\Psi^2+\B)(\log\Psi- k)]}\,.
\end{equation}
Since $\Psi^2\geq r_i^2\geq -\B>0$, choosing $k$ to be big enough we have that the right hand side of~\eqref{eq:first_square_bracket_term} is negative at any point in $\Omega$. As a consequence, for such $k$, from~\eqref{eq:elliptic_ineq_F_version2} we get
\begin{equation*}
\De F_\gamma\,-\,2
\frac{\Psi^4+4\B \Psi^2-\B^2-4\B\Psi^2(\log\Psi- k)}{(\Psi^2-\B)^2[2\Psi^2-(\Psi^2+\B)(\log\Psi- k)]}
\langle\na F_\gamma|\na u\rangle\,\geq\,0\,.
\end{equation*}
We can then apply the Maximum Principle to find out that $F_\gamma$ is nonpositive on the whole $\Omega$. This concludes the proof.
\end{proof}

The question still remains open of whether it is possible to remove the hypothesis $r_i^2\geq -\B$ in the above theorem. Looking at the proof, it is clear that the freedom in the choice of $k$ should allow to weaken that hypothesis. However, it does seem hard to remove it entirely. 

The rest of our arguments does not seem to depend heavily on the hypothesis $2a+\alpha^2\leq 2b+\beta^2$. In other words, even in the case $2a+\alpha^2> 2b+\beta^2$, our method should allow to conclude the desired rigidity, provided we have a model solution to compare with and the gradient estimate $W\leq W_0$ is in force.
In order to prove Theorem~\ref{thm:reichel} in full generality using our approach, these seem to be the two issues that are left to address.

%\bibliographystyle{plain}

%% BIBLIOGRAFIA CON BIBTEX %
\bibliographystyle{plain}
\bibliography{biblio}

\begin{thebibliography}{10}

\bibitem{Aftalion_Busca}
A.~Aftalion and J.~Busca.
\newblock Radial symmetry of overdetermined boundary-value problems in exterior
  domains.
\newblock {\em Arch. Rational Mech. Anal.}, 143(2):195--206, 1998.

\bibitem{ABM}
V.~Agostiniani, S.~Borghini, and L.~Mazzieri.
\newblock On the {S}errin problem for ring-shaped domains, 2021.
\newblock arXiv:2109.11255.

\bibitem{Agostiniani_Fogagnolo_Mazzieri}
V.~Agostiniani, M.~Fogagnolo, and L.~Mazzieri.
\newblock Sharp geometric inequalities for closed hypersurfaces in manifolds
  with nonnegative {R}icci curvature.
\newblock {\em Invent. Math.}, 222(3):1033--1101, 2020.

\bibitem{Agostiniani_Mazzieri_Oronzio}
V.~Agostiniani, L.~Mazzieri, and F.~Oronzio.
\newblock A {G}reen's function proof of the {P}ositive {M}ass {T}heorem, 2021.
\newblock arXiv:2108.08402.

\bibitem{Alessandrini}
G.~Alessandrini.
\newblock A symmetry theorem for condensers.
\newblock {\em Math. Methods Appl. Sci.}, 15(5):315--320, 1992.

\bibitem{Alexandrov}
A.~D. Alexandrov.
\newblock A characteristic property of spheres.
\newblock {\em Ann. Mat. Pura Appl. (4)}, 58:303--315, 1962.

\bibitem{Barbu_Enache}
L.~Barbu and C.~Enache.
\newblock A free boundary problem with multiple boundaries for a general class
  of anisotropic equations.
\newblock {\em Appl. Math. Comput.}, 362:124551, 9, 2019.

\bibitem{Barrios_Montoro_Sciunzi}
B.~Barrios, L.~Montoro, and B.~Sciunzi.
\newblock On the moving plane method for nonlocal problems in bounded domains.
\newblock {\em J. Anal. Math.}, 135(1):37--57, 2018.

\bibitem{Beig_Simon}
R.~Beig and W.~Simon.
\newblock On the uniqueness of static perfect-fluid solutions in general
  relativity.
\newblock {\em Comm. Math. Phys.}, 144(2):373--390, 1992.

\bibitem{Borghini}
S.~Borghini.
\newblock Static {B}lack {H}ole {U}niqueness for nonpositive masses, 2020.
\newblock arXiv:2008.09578.

\bibitem{BCM}
S.~Borghini, P.~T. Chru{\'s}ciel, and L.~Mazzieri.
\newblock On the uniqueness of schwarzschild-de sitter spacetime, 2019.
\newblock arXiv:1909.05941.

\bibitem{Borghini_Mazzieri_I}
S.~Borghini and L.~Mazzieri.
\newblock On the mass of static metrics with positive cosmological constant:
  {I}.
\newblock {\em Classical Quantum Gravity}, 35(12):125001, 43, 2018.

\bibitem{Borghini_Mazzieri_II}
S.~Borghini and L.~Mazzieri.
\newblock On the mass of static metrics with positive cosmological constant:
  {II}.
\newblock {\em Comm. Math. Phys.}, 377(3):2079--2158, 2020.

\bibitem{Brandolini_Nitsch_Salani_Trombetti}
B.~Brandolini, C.~Nitsch, P.~Salani, and C.~Trombetti.
\newblock On the stability of the {S}errin problem.
\newblock {\em J. Differential Equations}, 245(6):1566--1583, 2008.

\bibitem{Chrusciel_Simon}
P.~T. Chru\'{s}ciel and W.~Simon.
\newblock Towards the classification of static vacuum spacetimes with negative
  cosmological constant.
\newblock {\em J. Math. Phys.}, 42(4):1779--1817, 2001.

\bibitem{Ciraolo_Vezzoni}
G.~Ciraolo and L.~Vezzoni.
\newblock On {S}errin's overdetermined problem in space forms.
\newblock {\em Manuscripta Math.}, 159(3-4):445--452, 2019.

\bibitem{Enciso_Peralta-Salas}
A.~Enciso and D.~Peralta-Salas.
\newblock Symmetry for an overdetermined boundary problem in a punctured
  domain.
\newblock {\em Nonlinear Anal.}, 70(2):1080--1086, 2009.

\bibitem{Farina_Kawohl}
A.~Farina and B.~Kawohl.
\newblock Remarks on an overdetermined boundary value problem.
\newblock {\em Calc. Var. Partial Differential Equations}, 31(3):351--357,
  2008.

\bibitem{Farina_Valdinoci}
A.~Farina and E.~Valdinoci.
\newblock A pointwise gradient estimate in possibly unbounded domains with
  nonnegative mean curvature.
\newblock {\em Adv. Math.}, 225(5):2808--2827, 2010.

\bibitem{Fogagnolo_Mazzieri_Pinamonti}
M.~Fogagnolo, L.~Mazzieri, and A.~Pinamonti.
\newblock Geometric aspects of {$p$}-capacitary potentials.
\newblock {\em Ann. Inst. H. Poincar\'{e} Anal. Non Lin\'{e}aire},
  36(4):1151--1179, 2019.

\bibitem{Fogagnolo_Pinamonti}
M.~Fogagnolo and A.~Pinamonti.
\newblock New integral estimates in substatic riemannian manifolds and the
  {A}lexandrov theorem, 2021.
\newblock arXiv:2105.04672.

\bibitem{Garofalo_Lewis}
N.~Garofalo and J.~L. Lewis.
\newblock A symmetry result related to some overdetermined boundary value
  problems.
\newblock {\em Amer. J. Math.}, 111(1):9--33, 1989.

\bibitem{Gidas_Ni_Nirenberg}
B.~Gidas, W.~M. Ni, and L.~Nirenberg.
\newblock Symmetry and related properties via the maximum principle.
\newblock {\em Comm. Math. Phys.}, 68(3):209--243, 1979.

\bibitem{Gilbarg_Trudinger}
D.~Gilbarg and N.~S. Trudinger.
\newblock {\em Elliptic partial differential equations of second order}.
\newblock Classics in Mathematics. Springer-Verlag, Berlin, 2001.
\newblock Reprint of the 1998 edition.

\bibitem{Kamburov_Sciaraffia}
N.~Kamburov and L.~Sciaraffia.
\newblock Nontrivial solutions to {S}errin's problem in annular domains.
\newblock {\em Ann. Inst. H. Poincar\'{e} Anal. Non Lin\'{e}aire}, 38(1):1--22,
  2021.

\bibitem{Lee_Neves}
D.~A. Lee and A.~Neves.
\newblock The {P}enrose inequality for asymptotically locally hyperbolic spaces
  with nonpositive mass.
\newblock {\em Comm. Math. Phys.}, 339(2):327--352, 2015.

\bibitem{Ma_Liu}
L.~Ma and B.~Liu.
\newblock Symmetry results for decay solutions of elliptic systems in the whole
  space.
\newblock {\em Adv. Math.}, 225(6):3052--3063, 2010.

\bibitem{Magnanini}
R.~Magnanini.
\newblock Alexandrov, {S}errin, {W}einberger, {R}eilly: simmetry and stability
  by integral identities.
\newblock In {\em Bruno {P}ini {M}athematical {A}nalysis {S}eminar 2017},
  volume~8 of {\em Bruno Pini Math. Anal. Semin.}, pages 121--141. Univ.
  Bologna, Alma Mater Stud., Bologna, 2017.

\bibitem{Payne_Philippin}
L.~E. Payne and G.~A. Philippin.
\newblock On two free boundary problems in potential theory.
\newblock {\em J. Math. Anal. Appl.}, 161(2):332--342, 1991.

\bibitem{Pohozaev}
S.~I. Poho\v{z}aev.
\newblock On the eigenfunctions of the equation {$\Delta u+\lambda f(u)=0$}.
\newblock {\em Dokl. Akad. Nauk SSSR}, 165:36--39, 1965.

\bibitem{Reichel2}
W.~Reichel.
\newblock Radial symmetry by moving planes for semilinear elliptic {BVP}s on
  annuli and other non-convex domains.
\newblock In {\em Elliptic and parabolic problems ({P}ont-\`a-{M}ousson,
  1994)}, volume 325 of {\em Pitman Res. Notes Math. Ser.}, pages 164--182.
  Longman Sci. Tech., Harlow, 1995.

\bibitem{Reichel3}
W.~Reichel.
\newblock Radial symmetry for an electrostatic, a capillarity and some fully
  nonlinear overdetermined problems on exterior domains.
\newblock {\em Z. Anal. Anwendungen}, 15(3):619--635, 1996.

\bibitem{Reichel1}
W.~Reichel.
\newblock Radial symmetry for elliptic boundary-value problems on exterior
  domains.
\newblock {\em Arch. Rational Mech. Anal.}, 137(4):381--394, 1997.

\bibitem{Roncoroni}
A.~Roncoroni.
\newblock A {S}errin-type symmetry result on model manifolds: an extension of
  the {W}einberger argument.
\newblock {\em C. R. Math. Acad. Sci. Paris}, 356(6):648--656, 2018.

\bibitem{Serrin}
J.~Serrin.
\newblock A symmetry problem in potential theory.
\newblock {\em Arch. Rational Mech. Anal.}, 43:304--318, 1971.

\bibitem{Sirakov}
B.~Sirakov.
\newblock Symmetry for exterior elliptic problems and two conjectures in
  potential theory.
\newblock {\em Ann. Inst. H. Poincar\'{e} Anal. Non Lin\'{e}aire},
  18(2):135--156, 2001.

\bibitem{Soave_Valdinoci}
N.~Soave and E.~Valdinoci.
\newblock Overdetermined problems for the fractional {L}aplacian in exterior
  and annular sets.
\newblock {\em J. Anal. Math.}, 137(1):101--134, 2019.

\bibitem{Weinberger}
H.~F. Weinberger.
\newblock Remark on the preceding paper of {S}errin.
\newblock {\em Arch. Rational Mech. Anal.}, 43:319--320, 1971.

\bibitem{Whitney}
H.~Whitney.
\newblock Analytic extensions of differentiable functions defined in closed
  sets.
\newblock {\em Trans. Amer. Math. Soc.}, 36(1):63--89, 1934.

\bibitem{Willms_Gladwell_Siegel}
N.~B. Willms, G.~M.~L. Gladwell, and D.~Siegel.
\newblock Symmetry theorems for some overdetermined boundary value problems on
  ring domains.
\newblock {\em Z. Angew. Math. Phys.}, 45(4):556--579, 1994.

\end{thebibliography}

\end{document}